\documentclass[11pt]{article}
\usepackage{palatino}
\usepackage{aliascnt}
\usepackage{amsmath}
\usepackage{amsrefs}
\usepackage{amssymb}
\usepackage{amsthm}
\usepackage{algorithm}
\usepackage{algpseudocode}
\usepackage{tabularx}

\usepackage{bm}
\usepackage{bbm}

\usepackage{calrsfs}
\usepackage{color}

\usepackage[sans]{dsfont}

\usepackage{enumerate}
\usepackage{epsfig}
\usepackage{extarrows}
\usepackage[mathscr]{euscript}
\usepackage[symbol]{footmisc}
\usepackage{framed}
\usepackage{fullpage}
\usepackage[margin=1in]{geometry}
\usepackage{graphicx}
\usepackage{import}

\usepackage{multirow}

\usepackage{paralist}
\usepackage{psfrag}

\usepackage{subfig}
\usepackage[]{times}
\usepackage{tikz}

\usepackage{transparent}
\usepackage[normalem]{ulem}
\usepackage{url}
\usepackage{verbatim}
\usepackage{wrapfig}
\usepackage{xcolor}
\usepackage{xifthen}
\definecolor{Green}{rgb}{0.0,0.45,0.0}
\usepackage[colorlinks,linkcolor=blue,citecolor=red,bookmarksopen=false]{hyperref}

\def\NewTheorem#1#2{%
	\newaliascnt{#1}{thmm}
	\newtheorem{#1}[#1]{#2}
	\aliascntresetthe{#1}
	\expandafter\def\csname #1autorefname\endcsname{#2}
}

\numberwithin{equation}{section}

\NewTheorem{lem}{Lemma}
\NewTheorem{thm}{Theorem}
\NewTheorem{cor}{Corollary}

\NewTheorem{prop}{Proposition}

\theoremstyle{definition}
\NewTheorem{defin}{Definition}

\theoremstyle{remark}
\NewTheorem{remark}{Remark}
\theoremstyle{definition}
\NewTheorem{eg}{Example}

\title{A Note On Gaussian Random Fields \& Underlying Markov Processes
Through a Central Limit Theorem}

\author{Edward C.\ Waymire\thanks{Department of Mathematics,  Oregon State University, Corvallis, OR, 97331. 
{waymire@math.oregonstate.edu}}
}

\begin{document}
\maketitle

\begin{abstract} 
Various classes of  
 Gaussian random fields associated with transient Markov processes $Y$
 have been introduced in the  probability and mathematical physics literature.  The present paper is based on a natural
class of Gaussian random fields, termed universal Gaussian random fields
(UGRF),  for an underlying
Markov processes $X$, on a state space $(S,\mathcal{S})$ and having an
ergodic invariant
initial distribution $\pi$,
via a central limit theorem of Rabi Bhattacharya for appropriately 
scaled additive integral functionals
$\int_0^{nt}f(X(s))ds = \sum_{j=1}^n\int_{(j-1)t}^{jt}f(X(s))ds$ for $f\in1_\pi^\perp\equiv \{f\in L^2(S,\pi):\langle f,1\rangle_\pi=0\}$.  A Lamperti-type time change is 
introduced to obtain an infinite dimensional stationary Ornstein-Uhlenbeck evolution within a framework introduced in a classic
paper of K. It\^o. In particular it is shown that the It\^o's 
deterministic component vanishes under this time change, and It\^o's continuous regularity theory is applied. 
Connections with  
 GRFs associated with Markov processes $Y$ in a sense of Dynkin, and in a sense of
 Diaconis and Evans, respectively, are established 
under additional conditions on the infinitesimal generator $(A,\mathcal{D}(A))$
of the underlying Markov process $X$.
 \end{abstract}
\section{Introduction and Preliminaries}
Various classes of  
 Gaussian random fields (GRFs) associated with transient Markov processes $Y$
 have been introduced in the  probability and mathematical physics literature  by Dynkin \cite{Dynkin1, Dynkin2},
 Diaconis and Evans \cite{DiaconisEvans},
and Nelson \cite{Nelson2}, to mention a few. 
The basis for these theories, which is the main focus of this note, 
can generally be traced back to the early development of Euclidean field theory
and ideas and observations by Schwinger \cite{Schwinger},  Nakano \cite{Nakano} and  Symanzik \cite{Symanzik1}.
Namely, the centered GRF can be specified by using the Greens function
from quantum field theory to specify its covariance.  The operators in the
simplest cases, e.g., the massive free boson particle, can be more generally viewed
as infinitesimal generators of  killed symmetric Markov processes. Symmetry
and transience provide the requisite positive semidefinitenss 
and symmetry properties for the Greens function to serve as
a covariance function. 

 The 
Gaussian assumption is natural from the perspective that, apart
from centering, its entire distribution can be specified by the covariance.  Physicists viewed it as physically natural\footnote{The author thanks Leonard Gross for sharing this historical perspective.} 
for non-interacting (free) fields because the 
 Laplacian term of the simplest wave equation
could be viewed in Fourier space as a ``sum of infinitely many non-interacting quantum harmonic oscillators".
 Beyond this,
an insightful observation on 
 Wolpert's \cite{wolpert} Euclidean free field representation 
 is provided by Faris
 (\cite{Faris}, Sec 1.10), 
where it is noted that 
\lq\lq the central limit theorem (clt) causes 
a weighted sum of the local times of many particles 
to become a Gaussian field\rq\rq; see Proposition 1.9
of (\cite{Faris}, loc, cit.).  
Local times of 
Markov processes present 
 an alternative class of additive functionals to the \lq\lq integral\rq\rq functionals 
 for the Gaussian clt. 
  
 A theory for 
killed {\it recurrent} Markov processes $Y$ was also 
introduced by Diaconis and Evans in \cite{DiaconisEvans}, independently
motivated by developments in 
random matrix theory.  However, as 
remarked in (\cite{DiaconisEvans}, p. 869), the essential
difference between their approach and that of
Dynkin is that for a finite
state space and transient continuous time
Markov chain with infinitesimal generator
$A$, Dynkin would choose
the generalized inverse $-A^{-1}$ (described in next
paragraph below) to construct
the covariance,  while these authors use 
 $-A$. In other words, the former uses the 
Greens function associated with $-A$
and the latter uses the Dirichlet form 
associated with $-A$ whereas, in the 
ergodic case, Theorem \ref{dirformthm}
involves the Dirichlet form of $-A^{-1}$.
In this regard, the present note is based on a seemingly lesser known
class of GRFs suggested in \cites{RBfclt,DiaconisEvans}
and termed Universal Gaussian Random Fields
(UGRFs) here.  

The existence and uniqueness 
of  a generalized inverse
 $ A^{-1}$ as a well-defined linear
operator on its range ${\mathcal R}({ A})$ is 
straightforward in the present context.
Namely,
it follows from the
Hille-Yosida theorem for strongly continuous
semigroups that $ A$ is a densely defined
closed linear operator, so that one may express
 $L^2(S,\pi) =\mathcal{N}^\perp({ A})\oplus\mathcal{N}({ A})$,
 where $\mathcal{N}(A)$ denotes the null space of $A$.  
 The restriction of
$(A,\mathcal{D}(A))$ to
  a linear operator from 
$\mathcal{D}(A)\cap \mathcal{N}^\perp({ A})$
onto $\mathcal{R}(A)\subset\mathcal{N}({ A})^\perp$, 
 is a bijection, in fact, an isomorphism.  
 $ A^{-1}$ is the inverse of this bijection, and vanishes
 on $\mathcal{R}^\perp(A)$.
 
 \begin{remark}The terminology {\it underlying} ergodic Markov process $X$ is used in
in the context of UGRFs, while {\it associated} 
 Markov process $Y$ refers
to the contexts of \cite{Dynkin1} and \cite{DiaconisEvans},
respectively.
\end{remark}

UGRFs can be associated with ergodic 
Markov processes $X$ on a state space $(S,\mathcal{S})$ with invariant
initial distribution $\pi$
via a central limit theorem for appropriately 
scaled additive integral functionals
$\int_0^{nt}f(X(s))ds = \sum_{j=1}^n\int_{(j-1)t}^{jt}f(X(s))ds$  for $f\in1_\pi^\perp\equiv \{f\in L^2(S,\pi):\langle f,1\rangle_\pi=0\}$ due to 
\cite{RBfclt}; also see \cite{kipvar}, \cite{wayfclt}.  
Apart from the identification of these random fields in space-time limits,
the focus is on the development of connections and counterparts to the
aforementioned GRF  theories. In particular, it is shown that
 the natural choice for a
subclass of self-adjoint 
infinitesimal generators of ergodic Markov processes associated
with the UGRF
is the Dirichlet form of $-A^{-1}$; see
Theorem \ref{dirformthm} and Example \ref{dirformeg}, below. 
Interestingly, the result involves a mix
of  both \cite{Dynkin1} and \cite{DiaconisEvans} theories in the
use of (generalized) $A^{-1}$ on one hand, and the Dirichlet form
on the other. 
Moreover, it is shown in Theorem \ref{dynkinugrf} that the GRF associated with a subclass of killed ergodic time-reversible Markov processes in the sense of Dynkin  is a bounded
linear transformation of the stationary UGRF associated with the ergodic Markov process via Theorem \ref{statugrfthm}. 
The conditions required for Theorem \ref{dynkinugrf} are a bit milder than those
for Theorem \ref{dirformthm}, although both require symmetries. 

 For technical reasons involving the integrability
of the Greens function, Dynkin chose to index his 
random fields by $\sigma$-finite measures.  However,
we shall continue to use closed
subspaces of  $L^2(S,\mathcal{S},\pi)$ 
 for test functions indexing the
random fields.  The continuous regularity 
Corollary \ref{contregcor} and Theorem \ref{dynkinugrf} of the present paper 
indicates an advantage of this framework.
The following is a 
definition of generalized Gaussian random fields that
is adopted here\footnote{Some of the notation used here
is taken from the simplified treatment of It\^o's theory developed by \cite{Farisito}.} 
\begin{defin}\label{spacebm}
\begin{itemize}
\item[i.] A  generalized spatial
Gaussian random field $G$ over a 
separable Hilbert space $H$, 
$\langle\cdot,\cdot\rangle_H$ is a collection
$\{G(f):f\in H\}$ of zero mean
Gaussian
random variables defined on a probability space
$(\Omega, {\mathcal F}, P)$, $f\to G(f), f\in  H$,
being linear and continuous
in probability, and  such that 
\begin{equation*}
\text{cov}(G(f_1),G(f_2)) = \langle f_1,Cf_2\rangle_H, \ f_1,f_2\in H,
\end{equation*}
where $C$ is a bounded, self-adjoint, non-negative operator
on $H$, referred to as the spatial covariance operator.  In the
case that $C=I$ is the identity operator, $G$ is referred to
as Gaussian white noise (or isonormal 
random field) on $H$, and $G$ is typically denoted as
 $W$ in this instance.  
 \item[ii.] A space-time generalized Brownian motion
$\{B(f,t):f\in H, t\ge 0\}$ is an evolution of mean 
zero generalized Gaussian random fields
on a separable
 Hilbert space $H$ with covariance operator $C$ 
 on  $(\Omega,\mathcal{F},P)$ such that
for any $m\ge 1$, $f_1,\dots,f_m\in H$, $t\to (B(f_1,t),\dots,B(f_m,t)), t\ge 0$,
is an $m$-dimensional  Brownian motion, for which
\begin{equation}\label{cylwhite}
{\mathbb E}B(f_1,s)B(f_2,t) = (s\wedge t)\langle f_1,Cf_2\rangle_H, \  s,t\ge 0, f_1,f_2\in H.
\end{equation}
\end{itemize}
 \end{defin}
 \begin{remark}
 As noted in \cite{Farisito}, one may also view the spatial
 covariance operator $C$ as a map $C:H\to H^*$ such
 that for $g\in H$, the bounded linear functional $Cg$
 has value at $f\in H$ given by $\langle f,Cg\rangle_H$;
 also denoted by the \lq\lq matrix\rq\rq pairing $fCg$
 in \cite{Farisito}.  The existence of a generalized 
 Gaussian random field on $H$ for a given 
 spatial covariance operator  $C$ is essentially a
 consequence of Kolmogorov's existence theorem. 
 The resulting distribution of the random field is a probability measure
 on the product space $\mathbb{R}^{H}$ with the
 sigma-field generated by finite dimensional rectangles.
 A nice approach to the construction of a stationary
 space-time Gaussian random field, having a prescribed
 spatial covariance operator and temporally dissipative semigroup (generator),
  is given in
 \cite{Farisito} for It\^o's stationary Ornstein-Uhlenbeck GRF
to be introduced below.
 \end{remark}
It may be noted by the Cauchy-Schwarz inequality that a generalized Gaussian random
field $G$, as a map $f\to G(f)$ from $H$ to
$L^2(\Omega,{\mathcal F},P)$, is continuous, and therefore
continuous  in probability. Also, if $C={\mathbf I}$ then the map
is a Hilbert space isometry. Thus the term `isonormal'.

As a matter of notation it is helpful to recall the form of
the  Ornstein-Uhlenbeck infinitesimal generator in the $n$ dimensional case,
and in a form extendable to  infinite-dimensional Hilbert space $H$
as developed in \cite{ito} and denoted there as an OU($\mathcal{L}(H))$-process. 
 \begin{remark} It\^o uses
 $\mathcal{L}(E)$ to denote the
vector space of linear functionals on a real vector
space $E$, and generally refers to a stochastic process
temporally indexed by $\mathbb{R}$ defined on a probability space
$(\Omega,\mathcal{F},P)$ with values in $\mathcal{L}(E)$ as
 an $\mathcal{L}(E)$-process.
 \end{remark}
 The generic form of the OU infinitesimal generator is 
$$Lf(x) = \frac{1}{2}C\Delta f(x) + \langle\Gamma x,\nabla f(x)\rangle,$$
where  $\Gamma $ is a dissipative operator with negative eigenvalues, $C$ is
a symmetric, positive definite diffusion coefficient with 
square root $\sqrt{C}\equiv C^{\frac{1}{2}}$. The corresponding Ornstein-Uhlenbeck process $Z$ may be 
viewed as determined by the stochastic differential 
equation
\begin{equation}\label{ousde}
dZ(t) = \Gamma Z(t)dt + \sqrt{C}dW(t),
\end{equation}
where $W$ is standard Brownian motion.
In the case that the initial distribution is the mean
zero Gaussian invariant distribution,  
\begin{equation}\label{oupi}
\pi(dx) = (2\pi)^{-\frac{n}{2}}(\text{det}Q)^{-\frac{1}{2}}
\exp\{-\frac{1}{2}\langle x,Q^{-1}x\rangle\}, \quad \Gamma C+ C\Gamma^*
= -Q,
\quad x\in{\mathbb R}^n,
\end{equation}
the second condition is the equilibrium condition. 
The two parameters $C,\Gamma$, respectively 
governing fluctuation and dissipation, 
  make their appearance
 in the {\it equilibrium} or {\it fluctuation-dissipation covariance}\footnote{The appearance of 
 $Q$ may be viewed as an illustration of the
 so-called fluctuation-dissipation theorem of statistical
 physics, and referred to as noise covariance in \cite{Farisito}.}
operator $Q$ given by (\ref{oupi}) as follows.  Detailed balance
is the condition that  
\begin{equation}
\Gamma C = C\Gamma^*.
\end{equation}
Setting 
\begin{equation}Q= -\frac{1}{2}\Gamma^{-1}C
\end{equation}
yields the equilibrium condition in (\ref{oupi}).

The {\it Lamperti transformation} of the $n=1$ dimensional
 centered standard Brownian motion process $W$ 
 for scalar parameters  $\Gamma\equiv\gamma < 0$,
$C\equiv c>0$, is defined by the time change 
\begin{equation}
t\to {Z}(t) := e^{{-|\gamma|} t}\sqrt{\frac{c}{2|\gamma|}}
{W}(e^{2|\gamma| t}), \ t\ge 0,
\end{equation}
makes  ${Z}$ a {\it stationary} Ornstein-Uhlenbeck process
on $\mathbb{R}$  with parameters
$\gamma<0$ and $c>0$ in
(\ref{oupi}). 

 As alluded to above, It\^o provides an elegant infinite dimensional version of (\ref{ousde}) where $\Gamma$
is an infinitesimal generator of a strongly continuous semigroup of positive contractions on the underlying Hilbert
space. However, it should also be noted that 
in the general theory, It\^o also uncovers
a possible additional `deterministic component' in
the stochastic differential equation that uniquely
determines his
Ornstein-Uhlenbeck process in the sense of equivalence of
finite dimensional distributions.  Moreover, It\^o shows
that given the Ornstein-Uhlenbeck process one can
recover a Brownian random field, referred to as 
the {\it innovation process},  that drives the 
stochastic differential equation. 

The organization is as follows. 
 Section \ref{grfsec}  provides the essential elements of Bhattacharya’s
clt from \cite{RBfclt}.  
 A characterization of UGRF
space-time Brownian random fields $B$ as well as a Lamperti-type transformation
to stationary UGRFs $Z$ of the infinite dimensional Ornstein-Uhlenbeck type
as conceived by  It\^o in \cite{ito} are included in this section. 
In particular it is shown that the Brownian UGRF $B$ is the
innovation process for the stationary UGRF $Z$ in the sense
of It\^o, and that It\^o's deterministic component vanishes
for $Z$. 
In Section \ref{contregsec} It\^o's theory of continuous
regularity is applied to the Brownian and stationary 
UGRFs; see Corollary \ref{riggedZcor}.
The two aforementioned Theorems \ref{dynkinugrf},  \ref{dirformthm},
are then presented in Section \ref{mainsec}. Loosely speaking, 
Theorem \ref{dynkinugrf} identifies
a class of GRFs associated with transient Markov processes that can be obtained as bounded
linear transformation of a Lamperti-type time-changed UGRF.
This result was originally motivated by
Nelson’s theory of the massive free boson on the torus in which the infinitesimal generator $\Delta-m^2$  is
that of a 
killed ergodic Brownian motion at a constant exponential rate $m^2> 0$.
From the perspective of 
Dynkin’s theory one would simply use the 
Greens function $-(\Delta-m^2)^{-1}$ to furnish the 
covariance operator of an {\it assumed} Gaussian random field.   
 Assuming further that the generalized
inverse  $A^{-1}$ is also
an infinitesimal generator, Theorem \ref{dirformthm}
 shows that the (temporally frozen) Brownian UGRF 
$\{B(f,1):f\in1^\perp_\pi\}$ for the underlying  ergodic
Markov process $X=\{X(t):t\ge 0\}$ with generator $A$ via \cite{RBfclt}
and Theorem \ref{spacetimecov} below, 
 is also the Gaussian random field 
associated with the Markov process $Y=\{Y(t): t\ge 0\}$  
defined by the Dirichlet form $-2\langle f,A^{-1}h\rangle_\pi,
f,h\in1_\pi^\perp$; i.e., in the senses of \cites{Dynkin1,DiaconisEvans}
noted earlier for finite state Markov processes.  
Example \ref{dirformeg} illustrates this 
result.  Although unrelated to its consideration here, this  simple 
two-point example was used  by Leonard Gross \cite{gross}
in his groundbreaking discovery of logarithmic Sobolev inequalities
in quantum field theory, and his derivation of  Nelson's hypercontraction
inequalities \cite{Nelson} via a central limit theorem. 

The
central limit theorem of \cite{RBfclt} does not require self-adjointness of
the infinitesimal generator.
So, while self-adjointness is central to physics, it is not required to define the Brownian UGRF,
nor the Stationary UGRF.

\section{Gaussian Random Field Preliminaries and Notation}\label{grfsec}
The Gaussian random field construction considered here 
is the result of the 
following central limit of Rabi Bhattacharya; also see Corollary \ref{equiv}.
\begin{thm}{Bhattacharya \cite{RBfclt}}
\label{rabifcltthm}
Suppose that $X = \{X_t:t\ge0\}$ is a progressively measurable
continuous parameter
Markov process on a measurable state space $(S,{\mathcal S})$
starting from an ergodic invariant probability $\pi$, and defined
on a complete probability space $(\Omega,{\mathcal F},P_\pi)$.   Then for centered
$f\in L^2(S,\pi)$, i.e., $\int_Sfd\pi=0$,
 belonging to the range ${\mathcal R}({ A})$
of the  infinitesimal generator 
$( A,{\mathcal D}({ A}))$ of a strongly continuous
semigroup on 
$L^2(S,\pi)$, the sequence 
$\mathcal{I}_n(f,\cdot) =\{I_n(f,t)\equiv n^{-\frac{1}{2}}\int_0^{nt}f(X_s)ds: t\ge 0\}, 
n=1,2,\dots$ converges weakly in $C[0,\infty)$ to Brownian
motion $B(f) =\{B(f,t):t\ge0\}$ starting at $0$ with zero drift and 
diffusion coefficient 
\begin{equation}\label{diffcoeff}
c(f) =  -2\langle  Ag,g\rangle_\pi =-2\langle f,A^{-1}f\rangle_\pi =,\quad  Ag = f.
\end{equation}
\end{thm}

As already 
noted in (\cite{RBfclt}, Remark 2.1.1),
 using linearity of the integrals and the Cram\'{e}r-Wold device, 
if $f_1,\dots,f_m\in{\mathcal R}({ A})$,  then
 $(\mathcal{I}_n(f_1,\cdot),\dots \mathcal{I}_n(f_m,\cdot))$ converges weakly as $n\to\infty$
to $m$-dimensional Brownian motion with zero drift
and spatial covariance 
\begin{equation}\label{cltcov}
((c(f_i,f_j)))_{1\le i,j\le m} = ((-\langle Ag_i,g_j\rangle_\pi-\langle g_i, Ag_j\rangle_\pi))_{1\le i,j\le m},
\end{equation}
where $ Ag_i=f_i, 1\le i\le m$. 
In this regard,  the spatial covariance
operator $C$ on $\mathcal{R}(A)$ is given by 
the self-adjoint operator
\begin{equation}\label{rbcov}
C = -(A^{-1} + A{^{-1}}^*).
\end{equation}

Based on this,  the following  theorem provides
our essential starting point.

\begin{thm}{(Universal Gaussian Random Field (UGRF) Limit)}
\label{spacetimecov}
The sequence of space-time 
random fields $\{\mathcal{I}_n(\cdot,\cdot)\}_{n=1}^\infty$
converges in distribution to a space-time generalized
Brownian motion $\{B(f,t):f\in\mathcal{R}(A),t\ge 0\}$
with spatial covariance $C$ given by (\ref{rbcov})
in the sense that 
\begin{itemize}
\item For fixed $(f_1,\dots,f_m)\in\mathcal{R}(A)^m, m\ge 1$,
$(\mathcal{I}_n(f_1,\cdot),\dots,\mathcal{I}_n(f_m,\cdot)), n\ge 1$, converges weakly to $(B(f_1,\cdot),\dots,B(f_m,\cdot))\in C([0,\infty);\mathbb{R}^m)$,
\item For fixed but arbitrary 
 $t_1<\cdots<t_m, f_1,\cdots f_m\in\mathcal{R}(A), m\ge 1$, the $m$-dimensional random vector
 $(\mathcal{I}_n(f_1,t_1),\cdots,\mathcal{I}_n(f_m,t_m))\in\mathbb{R}^m$, 
converges in distribution to the Gaussian random vector $(B(f_1,t_1),\dots,B(f_m,t_m))$.
\end{itemize}
\end{thm}

 The following corollary, originally proven 
in \cite{kipvar},
is implied by Theorem \ref{rabifcltthm},
making the clt's of \cite{RBfclt}
and \cite{kipvar} equivalent in this
special case; see \cite{wayfclt}. 
 Since 
the adjoint $ A^*$  of a densely defined 
operator $ A$ on $L^2(S,\pi)$ is always closed,
the
strong continuity assumption involved in the application of the
Hille-Yosida theorem can be
avoided in the self-adjoint case to obtain
a densely defined closed operator $ A$
for which $ A^{-1}$ therefore exists as
the generalized inverse.   
Let also note  that one may check that  self-adjointness is maintained
by the generalized inverse of a self-adjoint operator; 
also see \cite{Majumdar}.

\begin{cor}\label{equiv} If $ A$ is self-adjoint then Theorem \ref{rabifcltthm}
is valid for  $f\in\mathcal{D}({ A^{-\frac{1}{2}}})\cap1_\pi^\perp
\supset\mathcal{R}(A)$. 
\end{cor}

The following assumptions on UGRFs and the underlying
{\it ergodic} Markov processes going forward.
 
\

\noindent{\bf Random Field Assumptions 1 ({\bf RFA1}):}
$ A$ is assumed to be a densely defined self-adjoint infinitesimal
generator  on $L^2(S,\pi)$ of a strongly continuous semigroup
for an ergodic invariant probability $\pi$
such
that $ A^{-1}$ is a bounded linear operator on
$\mathcal{R}(A)=1_\pi^\perp$.

\

The assumption 
{\bf RFA1} is essentially that $ A$ is the
 infinitesimal generator of a time-reversible
 strongly continuous
 semigroup on $L^2(S,\pi)$ for an ergodic invariant
 probability $\pi$, 
  for which $0$ is a simple,
 isolated eigenvalue in the spectrum of $ A$, i.e., 
 there is a spectral gap and hence geometric ergodicity;
 see \cite{RBfclt}. 
The UGRFs
fall within the  framework of temporally evolving
spatial generalized Gaussian random
fields over a Hilbert space as given in Definition \ref{spacebm}.
The following corollary places Theorem \ref{spacetimecov} 
in this framework. 
\begin{cor}\label{kcylcor}
 Under {\bf RFA1}, 
the UGRF limit is a space-time Brownian motion on $H=1_\pi^\perp$
with the spatial covariance operator 
\begin{equation}
C = -2 A^{-1}.
\end{equation}
\end{cor}

The following theorem is well-known in the Gaussian
random field theory but included here to highlight the
role of trace-class operators.  
by a {\it regular version} is
meant that  there is a $H$-valued 
random variable $Z\in L^2(\Omega:H)$, i.e.,
${\mathbb E}||Z||^2_H<\infty$,  such
that 
\begin{equation*}
Y(f) = \langle Z,f\rangle_H, \ f\in H.
\end{equation*}

\begin{thm}\label{regthm}
A zero mean generalized Gaussian random field $Y$ 
has a regular version
 on a separable Hilbert space $H$ with covariance
operator $C$ if and only if $C$ is trace-class, i.e., has finite trace.
\end{thm}
\begin{proof}  Let $\{e_j\}$ be a (denumerable) orthonormal basis for H.
If $Y(f) = \langle Z,f\rangle_H$, for an $H$-valued
random variable $Z$,
\begin{equation*}
\text{tr}(C) = \sum_j \langle e_j,Ce_j\rangle_H
= {\mathbb E}Y^2(f)
=\sum_j{\mathbb E}\langle Z,\varphi_j\rangle_H^2
= {\mathbb E}||Z||_H^2<\infty.
\end{equation*}
Conversely, noting that a self-adjoint trace class covariance operator $C$
implies that $C$ is compact, in fact Hilbert-Schmidt type, and therefore has an o.n. basis of
eigenfunctions $\{e_j\}$ with non-negative eigenvalues $\{\gamma_j\}$.
Note that $\mathbb{E}Y(e_j)Y(e_k) = \langle e_j,Ce_k\rangle_H
= \langle e_j,\gamma_ke_k\rangle_H=  \gamma_j\delta_{jk}$.
So $Y(e_j), j\ge 1$, is a sequence of independent Gaussian random
variables with mean zero and variance $\gamma_j, j\ge 1$.
 Define $Z=\sum_jY(e_j)e_j$.
Then $\mathbb{E}||Y||_H^2 = \sum_j\gamma_j <\infty$. 
Thus $||Z||_H<\infty$ a.s., so that $Z\in H$ a.s..  The distribution
of $\langle Z,f\rangle_H, f\in H,$ is a Gaussian probability measure on $H$ with
the covariance operator $C$, i.e., 
\begin{equation}\mathbb{E}[\langle Z,f_1\rangle_H
\langle Z,f_2\rangle_H] = \sum_j\langle f_1,e_j\rangle_H\langle f_2,
\gamma_je_j\rangle_H = \sum_j\langle f_1,e_j\rangle_H\langle f_2,Ce_j\rangle_H = \langle f_1,Cf_2\rangle_H.\end{equation}
\end{proof}

In the setting of a trace-class covariance operator
the continuity in probability
is elevated to an almost sure property.
\begin{cor}\label{traceclasscor}
 If, in addition to {\bf RFA1}, $ A^{-1}$
is trace-class on $1_\pi^\perp$, then
$f\to B(f), f\in 1_\pi^\perp$ is a.s. continuous.
\end{cor}
\begin{proof} Since $B$ is regular,  $f_n\to f$ in $H$ implies $\langle Z,f_n\rangle_H
\to\langle Z,f\rangle_H$ a.s. 
\end{proof}
Section \ref{contregsec} provides It\^o's \cite{ito}
approach to
creating continuous regular versions of the 
infinite dimensional Ornstein-Uhlenbeck processes
on a rigged Hilbert space (or Gelfand triple).  
Although different in details, this may be viewed
as a variant on the seminal work of \cite{grossberk}
for the construction of suitable spaces, referred
to as abstract Wiener spaces,  to accommodate
infinite dimensional processes.
\begin{remark}
As a matter of perspective, if one
considers Brownian motion then it is well-known
that Kolmogorov's product space is too large for
the measurability of events depending on uncountably
many time points, such as regularity of sample paths, while the space of absolutely continuous
functions is too small to support the distribution
of Brownian motion. The space of continuous functions
was first shown by Norbert Wiener to be the proper 
space, thus the name adopted by \cite{gross}; see
e.g., see \cite{stroockAWS} for a more detailed perspective.
\end{remark}

\subsection{Brownian UGRF}
In order to achieve some natural connections with It\^o's theory and/or
axioms of Euclidean field theory
it is required that the temporal evolution be extended from the positive half-line
to all of $\mathbb{R}$.  However this can be achieved quite naturally using
Kolmogorov extension theorem in the context
of UGRFs since the underlying Markov process $X=\{X(t):t\ge 0\}$ is a 
{\it stationary process}, i.e, temporally shift-invariant (translation invariant),
 under the invariant initial distribution $\pi$. Thus, $X$ may be
naturally extended to $\{\tilde{X}(t):t\in \mathbb{R}\}$ as follows:
For $t_1\le t_2\le\cdots\le t_m$, the finite dimensional distribution
of $(X(t_1),\dots,X(t_m))$ is defined by that of the time-translated
vector $(X(t_1-t_1),\dots,X(t_m-t_1)) = (X(0),\dots,X(t_m-t_1))$.
The double-sided extension will continue to be denoted $\tilde{X}$. 

\begin{thm}[Brownian Time Extension]\label{extendtime}
Let $\{B(f,t):f\in 1_\pi^\perp, t\in\mathbb{R}\}$ denote the
temporal extension of $B$ with finite dimensional
distributions defined by Gaussian weak-convergence limit 
\begin{equation}
\lim_{n\to\infty}\frac{1}{\sqrt{n}}\bigg(\int_{nt_1}^0f_1(\tilde{X}_{s_1})ds_1,\dots,\int_{nt_k}^0f_k(\tilde{X}_{s_k})ds_k,
 \int_{0}^{nt_{k+1}}f_{k+1}(\tilde{X}_{s_{k+1}})ds_{k+1},\dots,\int_0^{nt_m}f_m(\tilde{X}_{s_m})ds_m\bigg),
\end{equation}
for $t_1<t_2\cdots<t_k<0<t_{k+1}<\cdots< t_m$.
 Then ${B}$ is a space-time
centered 
Gaussian random field indexed by $1_\pi^\perp\times\mathbb{R}$
with 
 \begin{equation}
 \text{cov}(B(f_1,t_1),B(f_2,t_2))
 =
 \begin{cases} (|t_1|\wedge|t_2|)(-2\langle f_1, A^{-1}f_2\rangle_\pi)
 & \text{if}\quad t_1t_2 >0\\
 0 & \text{if}\quad t_1t_2\le 0.
 \end{cases}
 \end{equation}
 Moreover, 
 the path $t\to B(f,t)$ has an
 almost surely continuous version as  Brownian motion.
\end{thm}
\begin{proof}
To simplify the notation involved in the calculations,
let 
\begin{equation}I_a^b(j) = \int_a^bf_j(X_{s_j})ds_j,\quad 
\Delta_a^b(j) = B(f_j,b)-B(f_j,a).\end{equation}
Now, suppose $t_1<t_2\cdots<t_k<0<t_{k+1}<\cdots < t_m$.
For the double-sided extension of $X$, one makes  (integral) changes
 of variable of the form $u_j = s_j-nt_1$ to obtain 
 \begin{eqnarray}
 &&\frac{1}{\sqrt{n}}\bigg(I_{nt_1}^0(1),\dots,I_{nt_k}^0(k),
 I_{0}^{nt_{k+1}}(k+1),\dots,I_0^{nt_m}(m)\bigg)\nonumber\\
 &=&\frac{1}{\sqrt{n}}\bigg(\int_0^{-nt_1}(f_1(\tilde{X}_{u_1+nt_1})du_1,\dots,\int_{n(t_k-t_1)}^{-nt_1}f_k(\tilde{X}_{u_k+nt_1})du_k,\nonumber\nonumber\\
 &&\int_{-{nt_1}}^{n(t_{k+1}-t_1)}f_{k+1}(\tilde{X}_{u_{k+1}+nt_1})du_{k+1},\dots,\int_{-nt_1}^{n(t_m-t_1)}f_m(\tilde{X}_{u_m+nt_1})du_m\bigg)\nonumber\\
&=&\frac{1}{\sqrt{n}}\bigg(I_0^{-nt_1}(1),\dots,I_{n(t_k-t_1)}^{-nt_1}(k),
 I_{-{nt_1}}^{n(t_{k+1}-t_1)}(k+1),\dots,I_{-nt_1}^{n(t_m-t_1)}(m)\bigg)\nonumber\\
&=&\frac{1}{\sqrt{n}}\bigg(I_0^{-nt_1}(1),\dots,
I_0^{-nt_1}(k)
-I_0^{n(t_k-t_1)}(k),
 I_{0}^{n(t_{k+1}-t_1)}(k+1)\nonumber\\
&&-I_0^{-nt_1}(k+1),\dots,
I_{0}^{n(t_m-t_1)}(m)
-I_{0}^{-nt_1}(m)\bigg)\nonumber\\
&&\Rightarrow \bigg(\Delta_0^{-t_1}(1),\dots,
\Delta^{-t_1}_{t_k-t_1}(k),
\Delta^{t_{k+1}-t_1}_{-t_1}(k+1),\dots,
\Delta^{t_m-t_1}_{-t_1}(m)\bigg)\nonumber\\
&=&\bigg(B(f_1,t_1),B(f_2,t_2),\dots,B(f_m,t_m)\bigg). 
 \end{eqnarray}
 That is, the limiting joint finite dimensional
  distribution of $(B(f_1,t_1),\dots,B(f_m,t_m))$ is the Gaussian distribution of 
  $$(\Delta_0^{-t_1}(1),\Delta^{-t_1}_{t_2-t_1}(2),\dots,
  \Delta^{t_m-t_1}_{-t_1}(m))$$
  previously defined for $-t_1\ge 0, t_j-t_1\ge 0, j=2,\dots,m$. 
Thus, computing the expected values
 $\mathbb{E}\Delta^{-t_1}_0(1)\Delta^{-t_1}_{t_j-t_1}(j)$, for $j\le k$, 
 and  $\mathbb{E}\Delta^{-t_1}_0\Delta^{t_j-t_1}_{-t_1}(j)$,
 for  $j\ge k+1$,
 one arrives at 
  \begin{equation}\label{covform1}
\text{cov}(B(f_1,t_1),B(f_j,t_j))
  =\begin{cases} |t_j|
  (-2\langle f_1, A^{-1}f_j\rangle_\pi)\quad j\le k\\
 0\quad k+1\le j\le m.
   \end{cases}
  \end{equation}
Similarly, for $i\ge 2$, computing
 $\mathbb{E}\Delta^{-t_1)}_{t_i-t_1}(i)\Delta^{-t_1}_{t_j-t_1}(j)$, 
 for $2\le i<j\le k$,
 $\mathbb{E}\Delta^{-t_1}_{t_i-t_1}(i)\Delta^{t_j-t_1}_{-t_1}(j)$, 
 for $2\le i\le k<j$, and
 $\mathbb{E}\Delta^{t_i-t_1}_{-t_1)}(i)\Delta^{t_j-t_1}_{-t_1}(j)$, 
 for $k<i<j\le m$,
 one obtains
\begin{equation}\label{covform2}
  \text{cov}(B(f_i,t_i),B(f_j,t_j))
  =
 \begin{cases}|t_j|
  (-2\langle f_1, A^{-1}f_j\rangle_\pi),\  2\le i<j\le k\\
0,\  2\le i <k+1\le j\le m\\
 (t_i\wedge t_j)
  (-2\langle f_1, A^{-1}f_j\rangle_\pi),\ k < i\le j\le m.
  \end{cases}
    \end{equation}
 \end{proof}
Observing that $B(f,0)=0$, it follows
 that $B(f,-t) = -\{B(f,0)-B(f,-t)\}$ and 
 $B(f,t)= \{B(f,t)-B(f,0)\}$ are increments, i.e.,
the forms of (\ref{covform1}),(\ref{covform2}) reveal independent 
 increments.

\subsection{Stationary UGRF}

The stationary Ornstein-Uhlenbeck process is a
familiar choice of the Gaussian free field as a 
ground state, e.g., 
see \cite{newman}. 
In this subsection we
 construct a stationary UGRF from the class of 
infinite dimensional Ornstein-Uhlenbeck processes. 
To define a Stationary UGRF,  let us consider 
a Lamperti transformation of the Brownian UGRF $B$
under the assumptions of {\bf RFA1}.
Let
\begin{equation}
c(f) :=
2\langle f,-A^{-1}f\rangle_\pi\equiv
2||(-A)^{-\frac{1}{2}}f|^2_\pi, \ f\in1_\pi^\perp=
\mathcal{D}((-A)^{-\frac{1}{2}}).
\end{equation}
 \begin{thm}\label{statugrfthm}
 Let $\gamma<0$ and define
\begin{equation}Z(f,t) = 
 e^{-|\gamma|t}\frac{1}{\sqrt{2|\gamma|}}B(f,e^{2|\gamma|t})\equiv  
 e^{-|\gamma|t}\sqrt{\frac{c(f)}{2|\gamma|}}W(f,e^{2|\gamma|t}),
 \quad t\in\mathbb{R}, f\in
 1_\pi^\perp.\end{equation}
 Then, $t\to Z(\cdot,t)$ is a stationary random field such
 that for fixed $f_1,\dots,f_m\in1_\pi^\perp$,  one has that
 $\{Z(f_1,t),\dots,Z(f_m,t)):t\in\mathbb{R}\}$ is a stationary
 multivariate Ornstein-Uhlenbeck process with $\Gamma = \gamma{\mathbf I}$,
 and $C=-2A^{-1}$.
\end{thm}
\begin{proof}
For fixed $f_1,\dots,f_m\in1_\pi^\perp$, $\{(B(f_1,t),\dots,B(f_m,t)):t\in\mathbb{R}\}$ is a multivariate mean zero Brownian motion. Write 
\begin{equation}\label{scalebm}
(B(f_1,t_1),\dots,B(f_m,t_m))^{\prime} 
= q^{\frac{1}{2}}(f_1,\dots,f_m)W_m(t),
\end{equation}
where $W_m(\cdot)$ is $m$-dimensional standard Brownian motion,
and 
\begin{equation}\label{selfadjcov}
((c(f_i,f_j)))_{1\le i,j\le m} = ((-2\langle f_i,A^{-1}f_j\rangle))_{1\le i,j\le m}
\end{equation}
is the covariance matrix of $\{(B(f_1,t),\dots,B(f_m,t)):t\in\mathbb{R}\}$.
Then, using the Lamperti transformation, since each component 
of $W_m(\cdot)$ is a one-dimensional standard Brownian motion,
$Z_m(s) = e^{-|\gamma|s}\frac{1}{\sqrt{2|\gamma|}}W_m(e^{2|\gamma|s}), s\in\mathbb{R},$ is a two-sided
stationary Ornstein-Uhlenbeck process. The linear transformation
in (\ref{scalebm}) makes 
$\{e^{-|\gamma|s}\frac{1}{\sqrt{2|\gamma|}}(B(f_1,e^{2|\gamma|s}),\dots,B(f_m,e^{2|\gamma|s})):s\in\mathbb{R}\}
$ the family of
finite-dimensional distribution of stationary Ornstein-Uhlenbeck 
random field as asserted. 
\end{proof}
\begin{thm}{(An OU$\mathcal{L}(1_\pi^\perp)$-Process Limit)}
\label{itolimit}
Assume that $A$ satisifies {\bf RFA1}.
\begin{enumerate}
\item{} Then, the stationary UGRF $Z$ 
 in Theorem \ref{statugrfthm} is a version of the unique
OU$\mathcal{L}(1_\pi^\perp)$-process
whose It\^o {\it characteristics}, see \cite{ito},
 are given by
 (i) $p(f) = \sqrt{ \frac{1}{|\gamma|}
\langle f,(-A)^{-1}f\rangle_\pi},\  f\in1_\pi^\perp$,
and (ii) the characteristic (dissipative) operator $\Gamma = -|\gamma|{\mathbf I}$.
\item{} It\^o's deterministic component 
$Z^{(d)}$ is zero. 
\item{} It\^o's innovative process corresponding
to $Z$ is given by the Brownian time extension
$B$ defined in the statement of 
Theorem \ref{statugrfthm}.
\item{} Defining
\begin{equation}
U_n(f,t) = \frac{1}{\sqrt{2|\gamma|}}
\frac{e^{-|\gamma|t}}{\sqrt{n}}I_n(f,e^{2|\gamma|t}),\quad f\in 1_\pi^\perp, t\ge 0,
n\ge 1, 
\end{equation}
for arbitrary Lamperti parameter $\gamma<0$,
 the sequence of random fields
 $\{U_n\}_{n=1}^\infty$ converges under $P_\pi$ in
  finite dimensional distribution to 
It\^o's infinite dimensional Ornstein-Uhlenbeck 
$\mathcal{L}(1^\perp_\pi)$-process with characteristic
(dissipative) operator
$\Gamma =\gamma{\mathbf I}$, and $C=-2A^{-1}$. 
\end{enumerate}
\end{thm}
\begin{proof}
For (1), it is straightforward to check that the stationary
UGRF satisfies the requirements of 
(\cite{ito}, Definition 2.3). Namely, $Z$ is a centered,
Guassian, Markov, stationary, mean continuous 
$\mathcal{L}(1_\pi^\perp)$-process. 
For example, to check that $Z$ is temporally Markov simply
recall that $B$ has independent temporal increments
and is therefore temporally Markov. The Markov property
is clearly preserved under the Lamperti transformation.
The remaining properties required by It\^o's 
theory are
obvious since $\Gamma =-|\gamma|{\mathbf I}$ is a bounded operator.
For 1(i) one has 
\begin{equation}
p^2(f) = \mathbb{E}Z^2(f,t) = \frac{1}{2|\gamma|}
e^{-2|\gamma|t}e^{2|\gamma|t}2\langle f,(-A)^{-1}f\rangle_\pi
=  \frac{1}{|\gamma|}
\langle f,(-A)^{-1}f\rangle_\pi,
\end{equation}
and for 1(ii)
\begin{equation}
\mathbb{E}(Z(f,t)\vert\sigma(Z(h,s), s\le 0,h\in1_\pi^\perp))
= Z(g(f,t),0)
\end{equation}
so that 
\begin{equation}
\frac{e^{-|\gamma|}}{2|\gamma|}B(f,1)
=\frac{1}{2|\gamma|}B(g(f,t),1).
\end{equation}
Thus, by linearity of $f\to B(f,1)$, one has
$e^{-|\gamma|t}f=g(f,t)$, i.e., 
$S(t)f = e^{-\Gamma t}f, t\ge 0$ with infinitesimal
generator $\Gamma f= -|\gamma|f$.
In view of (\cite{ito}, Theorem 9.1) this
uniquely determines the stationary UGRF
$Z$ as It\^o's  infinite dimensional $OU\mathcal{L}(1_\pi^\perp)$-process in the sense
of equivalent finite dimensional distributions.
For the proof of (2), appeal to 
(\cite{ito}, Proposition 6.2) to see that
\begin{equation}
p^{(d)}(f,g)=\lim_{u\to\infty}p(e^{-|\gamma|u}f,e^{-|\gamma|u}g) =\lim_{u\to\infty}e^{-2|\gamma|u}p(f,g)=0.
\end{equation}
Thus the deterministic part of $Z$ vanishes. For the proof
of (3), denote the innovative process for $Z$ by
$\tilde{B}$ and its characteristic seminorm 
by $\tilde{b}$.  First note that
\begin{equation}
p(f,g)=\frac{p^2(f)+p^2(g)-p^2(f-g)}{2}
= \frac{\langle f,(-A)^{-1}g\rangle_\pi}{|\gamma|}.
\end{equation}
Thus,
\begin{equation}
\tilde{b}^2(f) = -2p(\Gamma f,f) = 2|\gamma|p^2(f)
= 2|\gamma|\frac{\langle f,(-A)^{-1}f\rangle_\pi}{2|\gamma|}
= 2\langle f,(-A)^{-1}f\rangle_\pi.
\end{equation}
On the other hand the characteristic seminorm $b$ of 
the Brownian UGRF $B$ is given
\begin{equation}
b^2(f) = \mathbb{E}(B(f,1)-B(f,0))^2 = 
-2\langle f,A^{-1}f\rangle_\pi.
\end{equation}
Thus $B$ is the innovation process associated with
$Z$.  
Finally for (4), the asserted convergence follows from Theorem \ref{spacetimecov}.
\end{proof}

\section{Continuously Regular Versions of  the Brownian UGRF $B$ and the Stationary UGRF $Z$ on a Rigged
Hilbert Space}
\label{contregsec}
In this section the It\^o continuous regularity theory is applied to the 
UGRFs $B$ and $Z$. For this purpose the It\^o theory
\cite{ito} will be reviewed under the more
specialized assumptions and notations  of the present paper. 

Let $D = \mathcal{D}(\Gamma)=1_\pi^\perp$. 
Then, one has
\begin{equation}
p(f)  :=\sqrt{\mathbb{E}Z^2(f,t)}= \sqrt{\frac{\langle f,-A^{-1}f\rangle_\pi}{|\gamma|}},\  f\in1_\pi^\perp.
\end{equation}
The polarization identity
\begin{equation}\label{pfgform}
p(f,g) = \frac{p^2(f)+p^2(g) - p^2(f-g)}{2}
= \frac{\langle f,(-A)^{-1}g\rangle_\pi}{|\gamma|}.
\end{equation}
induces a norm
on $D$ determined by 
\begin{equation}
r(f,g) = p((I-\Gamma)f,(I-\Gamma)g) 
= \frac{(1+|\gamma|)^2}{|\gamma|}\langle f,(-A)^{-1}g\rangle_\pi, \quad f,g\in 1^\perp_\pi, 
\end{equation}
such that $(D,r)$ is a separable Hilbert space and
$\mathcal{D}(\Gamma^2)$ is a dense subspace; 
see (\cite{ito}, Section 10).

Note that
$r(f) = (1+|\gamma|)p(f)$.
Let $\{d_n\}$ be an orthonormal basis of $(D,r)$.
such that $d_n\in\mathcal{D}(\Gamma^2) = 1_\pi^\perp$
for each $n$. Let $\{\lambda_n\}$ be sequence of 
positive numbers such that 
\begin{equation}
|\lambda|^2 = \sum_n\lambda_n^2 <\infty.
\end{equation}
Then the operator $K$ defined by its tensor representation
\begin{equation} K = \sum_n\lambda_n d_n\otimes d_n
\end{equation}
is a strictly positive definite Hilbert-Schmidt operator on $(D=1_\pi^\perp,r)$, with orthonormal eigenvector and
eigenvalue pairs $(d_n,\lambda_n), n\ge 1$, i.e., 
\begin{equation}
Kf = \sum_n\lambda_n\langle d_n,f\rangle_rd_n,
\quad f\in D. 
\end{equation} Define a separable Hilbert
space $(D_K,r_K)$ as follows:
\begin{eqnarray}
D_K&=&KD\\
r_K(f) &=& r(K^{-1}f)=\frac{1+|\gamma|}{\sqrt{|\gamma|}}\sqrt{\sum_n\frac{\langle f,(-A)^{-1}d_n\rangle^2_\pi}{\lambda_n^2}}\\
 r_K(f,g) &=& r(K^{-1}f,K^{-1}g) =  
 \frac{(1+|\gamma|)^2}{|\gamma|}\sum_n\frac{\langle f,(-A)^{-1}d_n\rangle_\pi\langle g,(-A)^{-1}d_n\rangle_\pi}{\lambda_n^2}.
\end{eqnarray}
Then 
\begin{equation}
r(f) \le cr_K(f), \quad c = \sup_n\lambda_n <\infty.
\end{equation}
The sequence $\{e_n = \lambda_nd_n\}$ is an 
orthonormal basis for $(D_K,r_K)$. Now, define
$(H,||\cdot||)$ to be the dual space of 
bounded linear functionals on $(D_K,r_K)$.
The triple $\{(D_K,r_K),(D,r),(H,||\cdot||)\}$ is a
{\it Gelfand triple}, or 
 {\it rigged Hilbert space} construction.
 The orthonormal
basis of $(H,||\cdot||)$ dual to $\{e_n\}$ is denoted
by $\{e_n^\prime\}$. 
\begin{equation}
\langle e_m^\prime,e_n\rangle = \delta_{m,n}.
\end{equation}

The following definition is given in \cite{ito} based on
this modification of the Hilbert space. 
\begin{defin}[It\^o's Continuous Regularity]\label{contregdef}
 A family $\{Y(t):t\in(-\infty,\infty)\}$
 of  $H$-valued random variables is referred to
as an $H$-process. 
An $H$-process $Y$ is called
sample continuous if, for each $\omega$,
 $t\to Y(t,\omega)$ is continuous 
for the norm of $H$.  A sample continuous process
$Y^\prime$ is called a version of $Y$ if $Y(t)=Y^\prime(t)$
a.s. for every  $t$. Let $Y$ be an $\mathcal{L}(D)$-process
and $\tilde{Y}(t)$ a regular version of $Y(t)$ for each $t$.
A continuous version of $\tilde{Y}$, if it exists, is called
a continuous regular version of $Y$.
An $H$-process $Y$ is called a centered Gaussian
process if $\{\langle Y(t)|f\rangle: f\in D_K, t\in(-\infty,\infty)\}$
form a centered Gaussian collection of random variables,
where  $\langle y|f\rangle$ denotes the
value of $y\in H$ at $f\in D_K$.
\end{defin}
The property of being a centered Gaussian process is
preserved by taking versions. 
The following is a corollary to Theorem \ref{statugrfthm}
using It\^o's theory for $\mathcal{L}(D)$-processes. 
\begin{cor}\label{contregcor}
As $\mathcal{L}(D)$ processes,
the Brownian time UGRF $B$  and the
stationary UGRF $Z$ possess continuous regular 
versions.
\end{cor}
\begin{proof}
On appeal to (\cite{ito}, Theorem 10.7), this follows from Theorem \ref{statugrfthm}.
\end{proof}
In particular 
 one has the following result
  in somewhat striking analogy with \cite{albeveriokus};
in the statement  the notation
 $\mathcal{N}(m,\Sigma)$ generically denotes
a Gaussian measure on a Hilbert space, the
dual space $H$ to ${D}_K$
in this case,  having
 mean vector $m$
and covariance operator $\Sigma$.
\begin{cor}\label{riggedZcor}
$Z$ is a sample continuous $H$-process which is
centered, Gaussian, Markov and stationary 
having invariant measure $\mathcal{N}(0,\frac{1}{2}\Gamma^{-1}C)$
and with transition probability 
\begin{equation}
p(t;f,\cdot) = \mathcal{N}(e^{-|\gamma| t}f,
(1-e^{-2|\gamma|t})\frac{1}{2}\Gamma^{-1}C)(\cdot),\quad
f\in D_K,
\end{equation}
where $\Gamma =-|\gamma|{\mathbf I}$, and
$C= 2(-A)^{-1}$. 
\end{cor}
\begin{proof}
This follows directly from Theorem \ref{statugrfthm} and
 (\cite{ito},Theorem 10.8) since It\^o's covariance functional
 $v(f,g) =  \{\frac{2\langle f,(-A)^{-1}g\rangle_\pi}{2|\gamma|}\}_{f,g\in D_K}$ defines the operator ${1\over 2}\Gamma^{-1}C$. Specifically, one has from (\ref{pfgform}) 
that
 \begin{equation}
 p(f,g) = 
 \frac{\langle f,(-A)^{-1}g\rangle_\pi}{|\gamma|}
=\langle f,\frac{1}{2}\Gamma^{-1}Cg\rangle_\pi,
\quad f,g\in D_K
  \end{equation}
  for the invariant distribution, 
  and in the notation of \cite{ito},
  \begin{eqnarray}
  v(f,g)&=&p(f,g)-p(e^{\Gamma t}f,e^{\Gamma t}g)\nonumber\\
  &=& (1-e^{-2|\gamma|t})p(f,g)\nonumber\\
&=& (1-e^{-2|\gamma|t})\langle f,  \frac{1}{2}\Gamma^{-1}Cg\rangle_\pi \quad f,g\in D_K,
  \end{eqnarray}
  for the transition probability. 
\end{proof}
\section{UGRFs Associated With Time-Reversible Markov Processes}\label{mainsec}
This final section establishes connections between UGRFs
and their {\it underlying} Markov process  with
certain GRFs
and their {\it associated} Markov processes via their (i) Greens function or
(ii) Dirichlet form, respectively. 
The first theorem identifies a  class of 
GRFs associated with killed
ergodic Markov processes in the sense of Dynkin 
in terms of linearly transformed Stationary UGRFs and
$OU\mathcal{L}(1_\pi^\perp)$-processes in the sense of It\^o \cite{ito}. 
Moreover, this class includes Nelson’s massive free boson
example on a torus as represented by Albeverio and Kusuoka 
in \cite{albeveriokus}.  The restrictions on $A$ for the following
theorem are much weaker than those of Theorem \ref{itolimit}
and Theorem \ref{dirformthm} below. 
\begin{thm}\label{dynkinugrf} 
Consider the Greens operator $C=-(A-m^2)^{-1}$ 
where $m^2\ge 0$ and $A$ is  a densely defined self-adjoint infinitesimal
generator  on $L^2(S,\pi)$ of a strongly continuous semigroup
for an ergodic invariant probability $\pi$
such
that $ A^{-1}$ is a bounded linear operator on the range
$\mathcal{R}(A)=1_\pi^\perp$.
 Define a bounded linear transformation $T:1_\pi^\perp\to 1_\pi^\perp$ 
by
\begin{equation}
Tf = 
(-A)^{\frac{1}{2}}(A-m^2)^{-\frac{1}{2}}f, \quad f\in 1_\pi^\perp.
\end{equation}
Let $Z$ denote the Stationary UGRF obtained
by a Lamperti transformation with dissipative 
operator $\Gamma={\mathbf I}$ 
and with covariance operator $C = -2A^{-1}$.
Then the linearly transformed Stationary UGRF
 $Z$ defined by $TZ$
 is It\^o's $OU\mathcal{L}(1_\pi^\perp)$-process  with 
parameters
$C=-(A-m^2)^{-1}$ and
and characteristic operator $\Gamma = {\bf I}$. 
In particular,  $Q=C$.
\end{thm}
\begin{proof}
First note that using the spectral theory for 
self-adjoint operators on a Hilbert space, e.g., \cite{riesz},
$T$ is a bounded linear operator for any $m^2\ge 0$. Moreover $T:1_\pi^\perp\to1_\pi^\perp$ since 
\begin{equation}0 = \langle A1,1\rangle_\pi = \langle (-A)^{\frac{1}{2}}1,(-A)^{\frac{1}{2}}1\rangle_\pi =||(-A)^{\frac{1}{2}}1||_\pi^2\end{equation}
implies that $(-A)^{\frac{1}{2}}1=0$. Thus, for $f\in1_\pi^\perp$,
$Tf\in_\pi^\perp$ follows from the definition of $T$ and
self-adjointness of $A^{\frac{1}{2}}$.
It suffices to compute $C$. Noting stationarity,
one has for arbitrary $t\in\mathbb{R}$,
\begin{eqnarray}
\mathbb{E}Z(Tf,t)Z(Tg,t) &=& 2\langle Tf,(-A)^{-1}Tg\rangle_\pi\nonumber\\
&=& 2\langle (-A)^{\frac{1}{2}}(A-m^2)^{-\frac{1}{2}}f,(-A)^{-1}(-A)^{\frac{1}{2}}(A-m^2)^{-\frac{1}{2}}g\rangle_\pi\nonumber \\
&=& 2\langle (A-m^2)^{-\frac{1}{2}}f,(-A)^{\frac{1}{2}}(-A)^{-\frac{1}{2}}
(A-m^2)^{-\frac{1}{2}}g\rangle_\pi \nonumber\\
&=& 2\langle f,-(A-m^2)^{-1}g\rangle_\pi, \quad f,g\in1_\pi^\perp.
\end{eqnarray}
Thus the covariance operator of $TZ$ is given by
the Greens function $ -(A-m^2)^{-1} = \frac{1}{2}{\bf I}^{-1}(-2)(A-m^2)^{-1}
$.
\end{proof}

For the second connection consider the temporally frozen Brownian UGRF $B=\{B(f,1):f\in1_\pi^\perp\}$.
We wish to formally note that under rather restrictive conditions on $A$, the
UGRF $B$ is associated with a Markov
process $Y$ via a specification of its Dirichlet form. This is a Diaconis-Evans \cite{DiaconisEvans}
 association with the UGRF for the generator $2A^{-1}$ 
of the UGRF's underlying Markov process  for finite state spaces
 as discussed in the introduction. 

 \begin{thm}\label{dirformthm} Assume that $(A,\mathcal{D}(A))$ satisfies ${\mathbf RFA1}$ and, moreover,
 that $A^{-1}$  is an
  infinitesimal generator on ${\mathcal D}(A^{-1})\subset 1_\pi^\perp$ of
 a Markov process on $S$.  Then the Gaussian random field 
 $\{B(f,1):f\in1_\pi^\perp\}$ is associated  
  with a 
 Markov process $Y$ on $S$
 via the specification
 of its covariance by the Dirichlet form of $Y$.  \end{thm}
 \begin{proof}
 
Let 
$Y=\{Y(t): t\ge 0\}$ denote the Markov process 
 with infinitesimal generator $2A^{-1}$ on $1_\pi^\perp$, and independent of $B$. 
 Then,
since the Dirichlet form $-2\langle f,A^{-1}g\rangle_\pi, f,g\in1_\pi^\perp,$ of $Y$
is the covariance of the Gaussian random field $\{B(f,1): f\in1_\pi^\perp\}$,
it follows that the UGRF $B$ is the GRF associated with the Markov 
process $Y$  via this
specification.  

\end{proof}

\begin{eg}\label{dirformeg}
Let $S=\{1,2\}$ and suppose that $X$ is the ergodic reversible 
continuous time Markov process
on $S$ with infinitesimal generator $A= \begin{bmatrix}-1 & 1\\
                                                          1 & -1\end{bmatrix}$.  
                                                          Then $\pi_1=\pi_2=\frac{1}{2}$ and $1_\pi^\perp=\{(f_1,f_2)\in\mathbb{R}^2: f_1=-f_2\}$.
                                                          Since $Af=-2f$ when restricted to $f\in1_\pi^\perp\subset L^2(S,\pi)$, 
$A^{-1}f=-\frac{1}{2}f = \frac{1}{4}Af, f\in1_\pi^\perp$, i.e., in matrix form
 $A^{-1}=   \begin{bmatrix}-\frac{1}{4} & \frac{1}{4}\\
                              \frac{1}{4} & -\frac{1}{4}\end{bmatrix}$.                                                     
  For the covariance one has $\mathbb{E}B(f,1)B(g,1) = -2\langle f,A^{-1}g\rangle_\pi
  = f_1g_1, f,g\in H=1_\pi^\perp$.
Let $Y$ be the Markov process on $S$ with infinitesimal generator
$2A^{-1}$ on $1_\pi^\perp\subset L^2(S,{\mathcal S},\pi)$. 
Then $Y$  is defined by  the 
 symmetric Dirichlet form $\mathcal{E}(f,g) =-2\langle f,A^{-1}g\rangle_\pi, f,g\in1_\pi^\perp$,
 that furnishes the covariance of the UGRF  
 $B(f,1)$. In particular, $Y$ is a  `speeded up' version of the jump process
 $X$ underlying the UGRF.                                 
\end{eg}

\begin{remark} As remarked in \cite{DiaconisEvans}, a Dynkin type isomorphism
theorem in terms of local times \cite{dynkiniso} is unclear for specifications of a GRF covariance by  
Dirichlet forms. 
Perhaps an alternative class of additive functionals of $Y$
could serve ? 
 For a multinomial $p(x_1,\dots,x_m)$
 and $f_1,\dots,f_m\in1_\pi^\perp$, and $f,g\in1_\pi^\perp$,
$\mathbb{E}(B(f,1)B(g,1)p(B(f_j,1),j=1,\dots,m)$ is computable 
 as a  function of the inner products
 weighted by the combinatorics of pairings of indices, e.g.,
\begin{equation}{\mathbb E}(B(f,1)B(g,1)\frac{1}{2}B^2(h,1))
= 2\langle f,A^{-1}g\rangle_\pi\langle h,A^{-1}h\rangle_\pi
+ 4\langle f,A^{-1}h\rangle_\pi\langle g,h\rangle_\pi.\end{equation}
However an expression in terms of functions of $B^2(h,1)$ and 
some independent
non-negative additive functionals $N(f)$, say, of the associated Markov process $Y$,
remains illusive if at all possible.
\end{remark}
\section{Acknowledgments} 
 The author wishes to thank Mat Titus, Bill Faris and 
Greg Lawler for corrections, comments and insights that greatly improved previous
versions of this note, and to Rabi Bhattacharya for suggesting the topic, and to Chuck Newman for some
encouraging comments during the formative stages.
The author assumes sole ownership of any remaining errors.

\bibliography{ugrf.bib}

\end{document}